\documentclass[preprint,review,10pt]{amsart}
\usepackage{amsmath,amssymb,amsfonts,xspace}
\newtheorem{theorem}{Theorem}[section]

\theoremstyle{definition}
\newtheorem{definition}[theorem]{Definition}
\newtheorem{example}[theorem]{Example}

\newtheorem{corollary}[theorem]{Corollary}

\theoremstyle{remark}

\numberwithin{equation}{section}

\begin{document}

\title{a more general framework than the $\delta$-primary hyperideals  }

\author{M. Anbarloei}
\address{Department of Mathematics, Faculty of Sciences,
Imam Khomeini International University, Qazvin, Iran.
}

\email{m.anbarloei@sci.ikiu.ac.ir }





\begin{abstract}
The $\delta$-primary hyperideal is a concept unifing the $n$-ary prime and $n$-ary primary hyperideals under one frame where $\delta$ is a function which assigns to each hyperideal $Q$ of  $G$ a hyperideal $\delta(Q)$ of the same hyperring with specific properties. In this paper, 
for a commutative Krasner $(m,n)$-hyperring $G$ with scalar identity $1$, we aim to introduce and study the notion of  $(t,n)$-absorbing $\delta$-semiprimary hyperideals which is a more general  structure than $\delta$-primary hyperideals.   We say that a proper hyperideal $Q$ of $G$ is an  $(t,n)$-absorbing $\delta$-semiprimary hyperideal if
whenever   $k(a_1^{tn-t+1}) \in Q$ for $a_1^{tn-t+1} \in G$, then there exist $(t-1)n-t+2$ of the $a_i^,$s whose $k$-product is in $\delta(Q)$.  Furthermore, we extend the concept to weakly $(t,n)$-absorbing $\delta$-semiprimary hyperideals. Several properties and characterizations of these classes of hyperideals are determined. In particular, after defining srongly weakly $(t,n)$-absorbing $\delta$-semiprimary hyperideals, we present the condition in which a weakly $(t,n)$-absorbing $\delta$-semiprimary hyperideal is srongly. Moreover, we show that $k(Q^{(tn-t+1)})=0$ where the weakly $(t,n)$-absorbing $\delta$-semiprimary hyperideal $Q$ is not $(t,n)$-absorbing $\delta$-semiprimary. Also, we investigate the
stability of the concepts  under intersection,  homomorphism and cartesian product of hyperrings.

\end{abstract}
\maketitle
\section{Introduction}
For the first time, the idea of 2-absorbing ideals as an extension  of prime ideals was presented by Badawi in \cite{n100}. The concept and it$^,$s generalizations have been widely studied by many researchers. The notion of  2-absorbing $\delta$-semiprimary ideals in a commutative ring which is an expansion of the 2-absorbing ideals was introduced by Celikel. From \cite{Celikel},  a proper ideal $I$ of a  commutative ring $R$ refers to a (weakly) 2-absorbing $\delta$-semiprimary ideal if $x,y,z \in R$ and ($0 \neq xyz \in I$) $xyz \in I$ imply $xy \in \delta(I)$ or $yz \in \delta(I)$ or $xz \in \delta(I)$ where   $\delta$ is a function that assigns to each ideal $I$  an ideal $\delta(I)$ of the same ring.

Study on $n$-ary algebras goes back to Kasner’s lecture \cite{s5} at a  scientific meeting in 1904.   In 1928,   the first paper  was written concerning the theory of $n$-ary groups by Dorente.    The $n$-ary structures  have been studied in  \cite{s9},\cite{s6}, \cite{l1},\cite{l2}, \cite{l3}, \cite{ma}, \cite{d1} and \cite{rev1}. There is a hyperring in which the addition is a hyperoperation, while the multiplication is an ordinary binary operation. The hyperring introduced by Krasner is called Krasner hyperring. The Krasner $(m,n)$-hyperring as a generalization of the Krasner hyperrings and a subclass of $(m,n)$-hyperrings was introduced by Mirvakili and Davvaz in \cite{d1}.  $(G, h, k)$, or simply $G$, is  a Krasner $(m, n)$-hyperring if: 
(1) $(G, h$) is a canonical $m$-ary hypergroup;
(2) $(G, k)$ is a $n$-ary semigroup;
(3) The  operation $k$ is distributive with respect to the  hyperoperation $h$ , i.e., for every $g^{i-1}_1 , g^n_{ i+1}, x^m_ 1 \in G$, and $1 \leq i \leq n$,
$k(g^{i-1}_1, h(x^m _1 ), g^n _{i+1}) = h(k(g^{i-1}_1, x_1, g^n_{ i+1}),\cdots, k(g^{i-1}_1, x_m, g^n_{ i+1}));$
(4) $0$ is a zero element of the $n$-ary operation $k$, i.e., for every $g^n_ 2 \in R$ , 
$k(0, g^n _2) = k(g_2, 0, g^n _3) = \cdots = k(g^n_ 2, 0) = 0$. 
A non-empty subset $H$ of $G$ is  a subhyperring of $G$ if $(H, h, k)$ is a Krasner $(m, n)$-hyperring. 
The non-empty subset $I$ of $G$ is a hyperideal if $(I, h)$ is an $m$-ary subhypergroup of $(G, h)$ and $k(g^{i-1}_1, I, g_{i+1}^n) \subseteq I$, for all $g^n _1 \in G$ and $1 \leq i \leq n$. 
By $g^j_i$ we mean the sequence $g_i, g_{i+1},\cdots, g_j$ and it is the empty symbol if $j< i$. Then we have
$h(a_1,\cdots, a_i, b_{i+1},\cdots, b_j, c_{j+1},\cdots, c_n)=h(a^i_1, b^j_{i+1},c^n_{j+1})$ and we write  $h(a^i_1, b^{(j-i)}, c^n_{j+1}),$ where $b_{i+1} =\cdots = b_j = b$. 
We define
$h(G^n_1) = \bigcup \{h(g^n_1) \ \vert \ g_i \in G_i, 1 \leq i \leq n \}$ where $G_1^n$ are  non-empty subsets  of $G$.
Recall from  \cite{sorc1} that a proper hyperideal $P$ of  $G$
is called $n$-ary prime if $k(G_1^ n) \subseteq P$ for hyperideals $G_1^n$ of $R$ implies that $G_i \subseteq P$ for some $1 \leq i \leq n$. If  $k(g^n_ 1) \in P$ for all $g^n_ 1 \in G$ and some hyperideal $P$ of $G$ implies that $g_i \in P$ for some $1 \leq i \leq n$, then  $P$  is an $n$-ary prime hyperideal of $G$, by Lemma 4.5 in \cite{sorc1}. 
Let $I$ be a hyperideal in  $G$ with
scalar identity $1$.  By $rad(I)$ we mean the intersection is taken over all $n$-ary prime hyperideals $P$ which contain $I$. $rad(I)=G$ if the set of all $n$-ary prime hyperideals containing $I$ is empty . It was shown (Theorem 4.23 in \cite{sorc1}) that if $g \in rad(I)$, then 
there exists $s \in \mathbb {N}$ such that $k(g^ {(s)} , 1^{(n-s)} ) \in I$ for $s \leq n$, or $k_{(l)} (g^ {(s)} ) \in I$ for $s = l(n-1) + 1$.
A proper hyperideal $I$ of $G$ with the scalar identity $1$ is  an $n$-ary primary hyperideal if $k(g^n _1) \in I$ and $g_i \notin I$ imply  $k(g_1^{i-1}, 1,g_{ i+1}^n) \in rad(I)$ for some $1 \leq i \leq n$. Theorem 4.28 in \cite{sorc1} shows that the radical of an $n$-ary primary hyperideal of $G$ is an $n$-ary prime hyperideal of $G$. Let $\mathcal{HI}(G)$ denote the set of all hyperideals of  $G$. Recall from \cite{mah3} that a function $\delta$  from $\mathcal{HI}(G)$ to $\mathcal{HI}(G)$ is  a hyperideal expansion of $G$  if   $I \subseteq \delta(I)$ and  if $I \subseteq J$ for any hyperideals $I, J$ of $G$, then $\delta (I) \subseteq \delta (J)$. For example, the mapping   $\delta_0$, $\delta_1$ and $\delta_R$, defined by $\delta_0(I)=I$,  $\delta_1(I)=rad(I)$ and $\delta_R(I)=R$ for all hyperideals $I$ of $G$, are hyperideal expansions of $G$. Furthermore, the function $\delta_q$,  defined by $\delta_q(I/J)=\delta(I)/J$ for any  hyperideal $I$ of $G$ containing hyperideal $J$ and an expansion function $\delta$ of $G$, is a hyperideal expansion  of $G/J$. A proper hyperideal $P$ of $G$ is an $(t,n)$-absorbing hyperideal as in \cite{rev2} if  $k(g_1^{tn-t+1}) \in P$ for some $g_1^{tn-t+1} \in G$ implies that there exist $(t-1)n-t+2$ of the $g_i^,$s whose $k$-product is in $P$. From the paper, a proper $P$ of $G$ is called an $(t,n)$-absorbing primary hyperideal if  $k(g_1^{tn-t+1}) \in P$ for some $g_1^{tn-t+1} \in G$ implies that $k(g_1^{(t-1)n-t+2}) \in P$ or a $k$-prodect of $(t-1)n-t+2$ of the $g_i^,$s except $k(g_1^{tn-t+2})$ is in $rad(P)$. 
 The notion of $n$-ary $\delta$-primary hyperideals was introduced in a Krasner $(m,n)$-hyperring in \cite{mah3}. This concept unifies the $n$-ary prime and $n$-ary primary hyperideals under one frame.  A proper hyperideal $Q$ of $G$ is called $\delta$-primary hyperideal if $k(a_1^n) \in Q$ for $a_1^n \in G$ implies that $a_i \in Q$ or $k(a_1^{i-1},1,a_{i+1}^n) \in \delta(Q)$ for some $1 \leq i \leq n$. Also, a proper hyperideal $Q$ of $G$ is  $(t,n)$-absorbing $\delta$-primary if for $a_1^{tn-t+1} \in G$, $k(a_1^{tn-t+1}) \in Q$ implies that $k(a_1^{(t-1)n-t+2}) \in Q$ or a $k$-product of $(t-1)n-t+2$ of $a_i^,$ s except $k(a_1^{(t-1)n-t+2})$ is in $\delta(Q)$.

 Now in this paper, we aim to define and study the notion of $(t,n)$-absorbing $\delta$-semiprimary hyperideals which is  more general  than the concept of   $\delta$-primary hyperideals in a Krasner $(m,n)$-hyperring. Additionally,  we present an extension of the notion  called weakly $(t,n)$-absorbing $\delta$-semiprimary hyperideals.
Among many results in this paper, it is shown (Theorem \ref{9}) that if  $Q$ is a weakly $(t,n)$-absorbing $\delta$-semiprimary hyperideal of $G$ that is not $(t,n)$-absorbing $\delta$-semiprimary, then $k(Q^{(tn-t+1)})=0$. Let  $Q$ be  a weakly $(t,n)$-absorbing $\delta$-semiprimary hyperideal of $G$ and $0 \neq k(Q_1^{tn-t+1}) \subseteq Q$ for   some hyperideals $Q_1^{tn-t+1}$ of $G$. It is shown (Theorem \ref{azad}) that  if  $Q$ is a free $\delta$-$(t,n)$-zero with respect to $k(Q_1^{tn-t+1})$, then   $k$-product of $(t-1)n-t+2$ of the $Q_i$  is a subset of $\delta(Q)$. Moreover, the stability of these concepts are examined under intersection,  homomorphism and cartesian product of hyperrings.
\section{(weakly) $(t,n)$-absorbing $\delta$-semiprimay hyperideals }
 Throughout this section, $G$ is a   commutative Krasner $(m,n)$-hyperring with scalar identity $1$. Initially, we
give the deﬁnition of $(t,n)$-absorbing $\delta$-semiprimay hyperideals of $G$.

\begin{definition}
Let $\delta$ be a hyperideal expansion of  $G$ and $t$ be a positive integer. A proper hyperideal $Q$ of $G$ is called an $(t,n)$-absorbing $\delta$-semiprimary hyperideal if
whenever  $k(a_1^{tn-t+1}) \in Q$ for $a_1^{tn-t+1} \in G$, then there exist $(t-1)n-t+2$ of the $a_i^,$s whose $k$-product is in $\delta(Q)$. 
\end{definition}

\begin{example} \label{inv}
For all $n$-ary prime hyperideal of $G$, we have

$n$-ary prime \hspace{0.3 cm} $\Longrightarrow $ \hspace{0.3 cm} $n$-ary $\delta$-primary \hspace{0.3 cm} $\Longrightarrow$ \hspace{0.3 cm} $(t,n)$-absorbing $\delta$-primary \hspace{0.3 cm}

 $\Longrightarrow$ 
\hspace{0.3 cm} $(t,n)$-absorbing $\delta$-semiprimary
\end{example}
The next example shows that the inverse of $"\Longrightarrow"^,$s in
 Example \ref{inv}, is not true, in general.
 \begin{example} 
In the Krasner $(2,2)$-hyperring $(G=[0,1],\boxplus, \circ)$  that $"\circ"$ is the usual multiplication on real numbers and 2-ary hyperoperation $"\boxplus"$ is defined by
\[
a \boxplus b=
\begin{cases}
\{max\{a, b\}\}, & \text{if $a \neq b$}\\
$[0,a]$, & \text{if $a =b,$}
\end{cases}\]

$\vspace{0.2cm}$

the hyperideal $Q=[0,0.5]$ is a $(2,2)$-absorbing $\delta_1$-semiprimary hyperideal of $G$ but it is not $2$-ary prime.
\end{example}
\begin{theorem} \label{1}
Let $Q$ be an $(t,n)$-absorbing $\delta$-semiprimary hyperideal of $G$ with $rad(\delta(Q)) \subseteq \delta(rad(Q))$. Then $rad(Q)$ is an $(t,n)$-absorbing  $\delta$-semiprimary hyperideal of $G$. 
\end{theorem}
\begin{proof}
Let $k(a_1^{tn-t+1}) \in rad(Q)$ for $a_1^{tn-t+1} \in G$ such that all products of $(t-1)n-t+2$ of the $a_i^,$s, other than $k(a_1^{(t-1)n-t+2})$, are not in $\delta(rad(Q))$. By the assumption, we conclude that  none of the $k$-products of the $a_i^,$s are in $rad(\delta(Q))$. 
From $k(a_1^{tn-t+1}) \in rad(Q)$, it follows that there exists $s \in \mathbb{N}$ with $k(k(a_1^{tn-t+1})^{(s)},1^{(n-s)}) \in Q$, for $s \leq n$ or $k_{(l)}(k(a_1^{tn-t+1})^{(s)}) \in Q$, for $s>n$ and $s=l(n-1)+1$. In the first possibility, we get $k(k(a_1)^{(s)},k(a_2)^{(s)},\cdots,k(a_{tn-t+1})^{(s)},1^{(n-s)}) \in Q$. Since $Q$ is an $(t,n)$-absorbing $\delta$-semiprimary hyperideal of $G$, we conclude that 

$\hspace{2cm}k(k(a_1)^{(s)},k(a_2)^{(s)},\cdots,k(a_{(t-1)n-t+2})^{(s)},1^{(n-s)})$   

$\hspace{2.5cm}=k(k(a_1^{(t-1)n-t+2})^{(s)},1^{(n-s)}) \in \delta(Q)$\\ because none of the $k$-products of the $a_i^,$s are in $rad(\delta(Q))$. Since $k(a_1^{(t-1)n-t+2}) \in rad(\delta(Q))$ and $rad(\delta(Q)) \subseteq \delta(rad(Q))$, then we have $k(a_1^{(t-1)n-t+2}) \in \delta(rad(Q))$. If $k_{(l)}(k(a_1^{tn-t+1})^{(s)}) \in Q$, for $s>n$ and $s=l(n-1)+1$, then we are done similarly. Thus  $rad(Q)$ is an $(t,n)$-absorbing  $\delta$-semiprimary hyperideal of $G$.
\end{proof}
The following result is a direct consequence of the previous theorem.
\begin{corollary}\label{2}
If  $Q$ is an $(t,n)$-absorbing $\delta_1$-semiprimary hyperideal of $G$, then $rad(Q)$ is an $(t,n)$-absorbing hyperideal of $G$. 
\end{corollary}
From \cite{sorc1},  the hyperideal generated by an element $g$ in $G$ is defined by 
$<g>=k(G,g,1^{(n-2)})=\{k(r,g,1^{(n-2)}) \ \vert \ r \in G\}.$
The following theorem will give us a characterization of $(t,n)$-absorbing $\delta$-semiprimary hyperideals.
\begin{theorem}
Every proper hyperideal is an $(t,n)$-absorbing $\delta$-semiprimary hyperideal of $G$ if and only if every proper principal hyperideal is an $(t,n)$-absorbing $\delta$-semiprimary hyperideal of $G$.
\end{theorem}
\begin{proof}
$\Longrightarrow$ It is obvious.

$\Longleftarrow $ Assume that $Q$ is a proper hyperideal of $G$ and $k(a_1^{tn-t+1}) \in Q$ for $a_1^{tn-t+1} \in G$. Therefore $k(a_1^{tn-t+1}) \in <k(a_1^{tn-t+1})>$. Since every proper principal hyperideal is an $(t,n)$-absorbing $\delta$-semiprimary hyperideal of $G$, there exist $(t-1)n-t+2$ of the $a_i^,$s whose $k$-product is in $\delta(<k(a_1^{tn-t+1})>) \subseteq \delta(Q)$. Hence $Q$ is an $(t,n)$-absorbing $\delta$-semiprimary hyperideal of $G$.
\end{proof}
Recall from \cite{mah3} that a hyperideal expansion  $\delta$  of  $G$ is called  intersection preserving if it satisfies $\delta(P\cap Q)=\delta(P) \cap \delta(Q)$, for all hyperideals $P$ and $Q$ of $G$. For example, hyperideal expansion  $\delta_1$  of $G$ is intersection preserving.
\begin{theorem} 
Let the hyperideal expansion $\delta$ of $G$ be intersection preserving. If  $Q_1^s$ are  $(t,n)$-absorbing $\delta$-semiprimary hyperideals of $G$ Such that  $\delta(Q_i) = P$ for each $1 \leq i \leq s$, then $Q=\bigcap_{i=1}^s Q_i$ is  an $(t,n)$-absorbing $\delta$-semiprimary hyperideal of $G$ with $\delta(Q)=P$.
\end{theorem}
\begin{proof}
Assume that   $k(a_1^{tn-t+1})\in Q$ for $a_1^{tn-t+1} \in G$ such that $k(a_1^{(t-1)n-t+2}) \notin \delta(Q)$. Since $\delta(Q)=\delta(\cap_{i=1}^s Q_i)=\cap_{i=1}^s \delta(Q_i)=P$, then there exists $1 \leq u \leq s$ such that $k(a_1^{(t-1)n-t+2}) \notin \delta(Q_u)$. Since $Q_u$ is  an $(t,n)$-absorbing $\delta$-semiprimary hyperideal of $G$ and $k(a_1^{tn-t+1})\in Q_u$, then there is a $k$-product of $(t-1)n-t+2$ of the $a_i^,$s is in $\delta(Q_u)=P=\delta(Q)$. Thus $Q=\bigcap_{i=1}^s Q_i$ is  an $(t,n)$-absorbing $\delta$-semiprimary hyperideal of $G$ with $\delta(Q)=P$.
\end{proof}
Let $(G_1, h_1, k_1)$ and $(G_2, h_2, k_2)$ be two Krasner $(m,n)$-hyperrings such that $1_{G_1}$ and $1_{G_2}$ be scalar identities of $G_1$ and $G_2$,  respectively. Then $(G_1 \times G_2,h=h_1 \times h_2, k=k_1 \times k_2) $ is a Krasner $(m, n)$-hyperring where

$h_1 \times h_2 ((a_1,b_1), \cdots,(a_m,b_m))=\{(a,b) \ \vert \  a \in h_1(a_1^m), b \in h_2(b_1^m)\},$

$k_1 \times k_2((a_1,b_1),\cdots,(a_n,b_n))=(k_1(a_1^n),k_2(b_1^n)),$\\
for all $a_i \in G_1$ and $b_i \in G_2$ \cite{cartis}. 
 \begin{theorem} \label{cart}
 Let $\delta_1$ and $\delta_2$ be two hyperideal expansions of Krasner $(m,n)$-hyperrings $G_1$ and $G_2$, respectively,   such that $\delta(Q_1 \times  Q_2)=\delta_1(Q_1) \times  \delta_2(Q_2)$ for  hyperideals $Q_1$ and $Q_2$ of $G_1$ and $G_2$, respectively. If $Q=Q_1 \times Q_2$ is an $(t+1,n)$-absorbing $\delta$-semiprimary hyperideal of $G=G_1  \times G_2$,
 then either  $Q_1$ is an $(t+1,n)$-absorbing $\delta_1$-semiprimary hyperideal of $G_1$ and $\delta_2(Q_2)=G_2$ or $Q_2$ is an $(t+1,n)$-absorbing $\delta_2$-semiprimary hyperideal of $G_2$ and $\delta_1(Q_1)=G_1$ or $Q_i$ is an $(t,n)$-absorbing $\delta_i$-semiprimary hyperideal of $G_i$ for each $i \in \{1,2\}$.
\end{theorem}
\begin{proof}
Let $Q=Q_1 \times Q_2$ be an $(t+1,n)$-absorbing $\delta$-semiprimary hyperideal of $G=G_1  \times G_2$. Assume that $\delta_1(Q_1) \neq G_1$ and $\delta_2(Q_2)=G_2$. Let us suppose that $k_1(a_1^{(t+1)n-t}) \in Q_1$ for some $a_1^{(t+1)n-t} \in G_1$ such that all products of $tn-t+1$ of the $a_i^,$s except $k_1(a_1^{tn-t+1})$ are not in $\delta(Q_1)$. Note that $k((a_1,0),\cdots,(a_{(t+1)n-t},0)) \in Q$ and all products of $tn-t+1$ of the $(a_i,0)^,$s are not in $\delta(Q)$. Since $Q$ is an $(t+1,n)$-absorbing $\delta$-semiprimary hyperideal of $G$, we get $k((a_1,0),\cdots,(a_{tn-t+1},0)) \in \delta(Q)=\delta_1(Q_1) \times \delta_2(Q_2)$ which means $k_1(a_1^{tn-t+1}) \in \delta(Q_1)$. Thus $Q_1$ is an $(t+1,n)$-absorbing $\delta_1$-semiprimary hyperideal of $G_1$. Similiar for the second assertion. For the third assertion, assume $\delta_1(Q_1) \neq G_1$ and $\delta_2(Q_2) \neq G_2$. Moreover, let us suppose that $Q_1$ is not an $(t,n)$-absorbing $\delta$-semiprimary hyperideal of $G_1$ and $k_1(a_1^{tn-t+1}) \in Q_1$. We define the following elements of $G$: $x_1=(a_1,1_{G_2}), x_2=(a_2,1_{G_2}), \cdots, x_{tn-t+1}=(a_{tn-t+1},1_{G_2}),x_{(t-1)n-t+2}=(1_{G_1},0)$. Therefore we have $k(x_1^{(t-1)n-t+2})=(k_1(a_1^{tn-t+1}),0) \in Q$, $k(x_1^{tn-t+1})=(k_1(a_1^{tn-t+1}),1_{G_2}) \notin \delta(Q)$ and $k(x_1,\cdots,\hat{x}_i,\cdots,x_{(t-1)n-t+2})=(k_1(a_1,\cdots,\hat{a}_i,\cdots,a_{(t-1)n-t+2}),0) \notin \delta(Q)$ for some $1 \leq i \leq tn-t+1$, a contradiction. Thus $Q_1$ is an $(t,n)$-absorbing $\delta_1$-semiprimary hyperideal of $G_1$. Similarly, we conclude that $Q_2$ is an $(t,n)$-absorbing $\delta_2$-semiprimary hyperideal of $G_2$
\end{proof}
 \begin{theorem} \label{cart3}
 Let $\delta_1,\cdots, \delta_{tn-t+1}$ be hyperideal expansions of Krasner $(m,n)$-hyperrings $G_1,\cdots,G_{tn-t+1}$  such that $\delta(Q_1 \times \cdots \times Q_{tn-t+1})=\delta_1(Q_1) \times \cdots \times \delta_{tn-t+1}(Q_{tn-t+1})$ for  hyperideals $Q_1,\cdots,Q_{tn-t+1}$ of $G_1,\cdots,G_{tn-t+1}$, respectively. If $Q=Q_1 \times \cdots \times Q_{tn-t+1}$ is an $(t+1,n)$-absorbing $\delta$-semiprimary hyperideal of $G=G_1 \times ... \times G_{tn-t+1}$, then either $Q_u$ is an $(t+1,n)$-absorbing $\delta$-semiprimary hyperideal of $G_u$ for some $1 \leq u \leq tn-t+1$ and $\delta_i(Q_i)=G_i$ for each $1 \leq i \leq tn-t+1$ and $i \neq u$ or $Q_u$ and $Q_v$ are $(t,n)$-absorbing $\delta_{u,v}$-semiprimary hyperideals of $G_u$ and $G_v$, respectively, for some $u,v \in \{1,\cdots, tn-t+1\}$ and $ \delta_i(Q_i)=G_i$ for all $1 \leq i \leq tn-t+1$ but $i \neq u,v$. 
 \end{theorem}
\begin{proof}
It can be  seen that the idea is true in a similar manner to the proof of Theorem \ref{cart}.
\end{proof}
Now, we want to extend the notion of $(t,n)$-absorbing $\delta$-semiprimary hyperideals to weakly  $(t,n)$-absorbing $\delta$-semiprimary hyperideal. Although different from each other in many aspects, they share quite a number of similar properties as well.
\begin{definition}
Let $\delta$ be a hyperideal expansion of  $G$ and $t$ be a positive integer. A proper hyperideal $Q$ of $G$ refers to  a weakly  $(t,n)$-absorbing $\delta$-semiprimary hyperideal if $a_1^{tn-t+1} \in G$ and   $0 \neq k(a_1^{tn-t+1}) \in Q$, then there exist $(t-1)n-t+2$ of the $a_i^,$s whose $k$-product is in $\delta(Q)$. 
\end{definition}
\begin{example} \label{classes} 
Suppose that $\mathbb{Z}_{12}$ is the set of all congruence classes of integers modulo $12$ and $H=\{1,5,7,11\}$ is multiplicative subgroup of units $\mathbb{Z}_{12}$. Construct $G$ as $\mathbb{Z}_{12}/H$. Then we have $G=\{\bar{0},\bar{1},\bar{2},\bar{3},\bar{4},\bar{6}\}$ in which $\bar{0}=\{0\}$, $\bar{1}=\{1,5,7,11\}$, $\bar{2}=\bar{10}=\{2,10\}$, $\bar{3}=\bar{9}=\{3,9\}$, $\bar{4}=\bar{8}=\{4,8\}$, $\bar{6}=\{6\}$. Consider Krasner hyperring $(G,\oplus,\star)$ that for all $\bar{a},\bar{b} \in G$, $\bar{a} \star \bar{b}=\overline{ab}$ and 2-ary hyperoperation $\oplus$ is defined as follows:

\[\begin{tabular}{|c|c|c|c|c|c|c|} 
\hline $\oplus$ & $\bar{0}$ & $\bar{1}$ & $\bar{2}$ & $\bar{3}$ & $\bar{4}$ & $\bar{6}$
\\ \hline $\bar{0}$ & $\bar{0}$ & $\bar{1}$ & $\bar{2}$ & $\bar{3}$& $\bar{4}$ & $\bar{6}$
\\ \hline $\bar{1}$ & $\bar{1}$ & $\bar{0},\bar{2},\bar{\bar{4}},\bar{6}$ & $\bar{1},\bar{3}$ & $\bar{2},\bar{4}$ & $\bar{1},\bar{3}$ & $\bar{1}$
\\ \hline $\bar{2}$ & $\bar{2}$ & $\bar{1},\bar{3}$ & $\bar{0},\bar{4}$ & $\bar{1}$ & $\bar{2},\bar{6}$ & $\bar{4}$
\\ \hline $\bar{3}$ & $\bar{3}$ & $\bar{2},\bar{4}$ & $\bar{1}$ & $\bar{0},\bar{6}$ & $\bar{1}$ & $\bar{3}$
\\ \hline $\bar{4}$ & $\bar{4}$ & $\bar{1},\bar{3}$ & $\bar{2},\bar{6}$ & $\bar{1}$ & $\bar{0},\bar{4}$ & $\bar{2}$
\\ \hline $\bar{6}$ & $\bar{6}$ & $\bar{1}$ & $\bar{4}$ & $\bar{3}$ & $\bar{2}$ & $\bar{0}$
\\ \hline
\end{tabular}\]

$\vspace{0.3cm}$

It is easy to see that the  hyperideal $Q=\{\bar{0},\bar{2},\bar{4},\bar{6}\}$ of $G$ is a $(2,2)$-absorbing $\delta_1$-semiprimary. 
\end{example}
\begin{theorem}
 If $Q$  is  a (weakly) $(t,n)$-absorbing $\delta$-semiprimary hyperideal of $G$, then $Q$ is (weakly) $(v,n)$-absorbing $\delta$-semiprimary for all $v>n$. 
 \end{theorem}
 \begin{proof}
 By using an argument similar to that in the proof of Theorem 4.4 in \cite{rev2}, one can  complete the proof.
 \end{proof}
\begin{theorem}\label{3}
Let $Q$ be a proper hyperideal of $G$. If $\delta(Q)$ is a (weakly) $(t,n)$-absorbing hyperideal of $G$, then $Q$ is a (weakly) $(t,n)$-absorbing $\delta$-semiprimary  hyperideal of $G$.
\end{theorem}
\begin{proof}
Let ($0 \neq k(a_1^{tn-t+1}) \in Q$) $k(a_1^{tn-t+1}) \in Q$ such  that all products of $(t-1)n-t+2$ of the $a_i^,$s, other than $k(a_1^{(t-1)n-t+2})$, are not in $\delta(Q)$.  Since $\delta(Q)$ is a (weakly) $(t,n)$-absorbing hyperideal of $G$ and $Q \subseteq \delta(Q)$, we conclude that  $k(a_1^{(t-1)n-t+2}) \in \delta(Q)$. This shows that $Q$ is a (weakly) $(t,n)$-absorbing $\delta$-semiprimary  hyperideal of $G$.
\end{proof}
\begin{theorem}\label{4}
Let $Q$ be a proper hyperideal of $G$ such that $\delta(\delta(Q))=\delta(Q)$. Then  $\delta(Q)$ is a (weakly) $(t,n)$-absorbing hyperideal of $G$ if and only if  $\delta(Q)$ is a (weakly) $(t,n)$-absorbing $\delta$-semiprimary  hyperideal of $G$.
\end{theorem}
\begin{proof}
$\Longrightarrow $ Assume that $\delta(Q)$ is a (weakly) $(t,n)$-absorbing hyperideal of $G$. Since $\delta(\delta(Q))=\delta(Q)$, we are done by Theorem \ref{3}.

$\Longleftarrow$ Let $\delta(Q)$ be a (weakly) $(t,n)$-absorbing $\delta$-semiprimary  hyperideal of $G$. Suppose that ($0 \neq k(a_1^{tn-n+1}) \in \delta(Q)$) $k(a_1^{tn-n+1}) \in \delta(Q)$. Since $\delta(Q)$ is  a (weakly) $(t,n)$-absorbing $\delta$-semiprimary  hyperideal of $G$, then there exist $(t-1)n-t+2$ of the $a_i^,$s whose $k$-product is in $\delta(\delta(Q))$. Since $\delta(\delta(Q))=\delta(Q)$, then the $k$-product of the $(t-1)n-t+2$ of the $a_i^,$s   is in $\delta(Q)$ which means $\delta(Q)$ is a (weakly) $(t,n)$-absorbing hyperideal of $G$.
\end{proof}
\begin{theorem}\label{5}
Let $Q$ be a (weakly) $(t,n)$-absorbing $\delta$-semiprimary hyperideal of $G$ and $P$ be a proper hyperideal of $G$ such that $P \subseteq Q$.
 If $\delta(Q)=\delta(P)$, then $P$ is a (weakly) $(t,n)$-absorbing $\delta$-semiprimary hyperideal of $G$.
\end{theorem}
\begin{proof}
Assume that ($0 \neq k(a_1^{tn-t+1}) \in P$) $k(a_1^{tn-t+1}) \in P$ for $a_1^{tn-t+1} \in G$. By the assumption, we get ($0 \neq k(a_1^{tn-t+1}) \in Q$) $k(a_1^{tn-t+1}) \in Q$ which implies  there exist $(t-1)n-t+2$ of the $a_i^,$s whose $k$-product is in $\delta(Q)$ because $Q$ is a (weakly) $(t,n)$-absorbing $\delta$-semiprimary hyperideal of $G$. From $\delta(Q)=\delta(P)$, it follows that the $k$-product of  $(t-1)n-t+2$ of the $a_i^,$s is in $\delta(P)$ which means $P$ is a (weakly) $(t,n)$-absorbing $\delta$-semiprimary hyperideal of $G$.
\end{proof}
\begin{definition}
Let $Q$ be a proper hyperideal of $G$. $Q$ refers to a strongly (weakly) $(t,n)$-absorbing $\delta$-semiprimary hyperideal if ($0 \neq k(Q_1^{tn-t+1}) \subseteq Q$) $k(Q_1^{tn-t+1}) \subseteq Q$ for some hyperideals $Q_1^{tn-t+1}$ of $G$, then there exist $(t-1)n-t+2$ of $Q_i^,$s whose $k$-product is a subset of $\delta(Q)$.
\end{definition}
\begin{definition}
Assume that $G$ is a commutative Krasner $(m,2)$-hyperring and $Q$ is a weakly $(2,2)$-absorbing $\delta$-semiprimary hyperideal of $G$. Then $(x,y,z)$ is said to be an $\delta$-$(2,2)$-zero of $Q$ for some $x,y,z \in G$ if $k(x,y,z)=0$, $k(x,y) \notin \delta(Q)$,  $k(y,z) \notin \delta(Q)$ and $k(x,z) \notin \delta(Q)$.
\end{definition}
\begin{theorem} \label{str}
Let  $G$ be a commutative Krasner $(m,2)$-hyperring, $Q$  a weakly $(2,2)$-absorbing $\delta$-semiprimary hyperideal of $G$ and $k(Q_1,x,y) \subseteq Q$ for some $x,y \in G$ and a hyperideal $Q_1$ of $G$. If $(q,x,y)$ is not a $\delta$-$(2,2)$-zero of $Q$ for all $q\in Q_1$ and $k(x,y) \notin \delta(Q)$, then $k(Q_1,x) \subseteq \delta(Q)$ or $k(Q_1,y) \subseteq \delta(Q)$.
\end{theorem}
\begin{proof}
Let $k(Q_1,x,y) \subseteq Q$ for some $x,y \in G$ and a hyperideal $Q_1$ of $G$ but $k(x,y) \notin \delta(Q)$, $k(Q_1,x) \nsubseteq \delta(Q)$ and $k(Q_1,y) \nsubseteq \delta(Q)$. Then we have $k(q_1,x) \nsubseteq \delta(Q)$ and $k(q_2,y) \nsubseteq \delta(Q)$ for some $q_1, q_2 \in Q_1$. Since $(q_1,x,y)$ is not a $\delta$-$(2,2)$-zero of $Q$ and $k(q_1,x,y) \in Q$, we get $k(q_1,y) \in \delta(Q)$. Similarly, we have $k(q_2,x) \in \delta(Q)$. Note that $k(h(q_1,q_2,0^{(m-2)}),x,y)=h(k(q_1,x,y),k(q_2,x,y),0^{(m-2)}) \subseteq Q$. Then we obtain $k(h(q_1,q_2,0^{(m-2)}),x)=h(k(q_1,x),k(q_2,x),0^{(m-2)}) \subseteq \delta(Q)$ or $k(h(q_1,q_2,0^{(m-2)}),y)=h(k(q_1,y),k(q_2,y),0^{(m-2)}) \subseteq \delta(Q)$. This follows that $k(q_1,x) \in h(-k(q_2,x),0^{(m-1)}) \subseteq \delta(Q)$ or $k(q_2,y) \in h(-k(q_1,y),0^{(m-1)}) \subseteq \delta(Q)$ which both of them are a contradiction. Consequently, $k(Q_1,x) \subseteq \delta(Q)$ or $k(Q_1,y) \subseteq \delta(Q)$.
\end{proof}
\begin{theorem} \label{str2}
Let  $G$ be a commutative Krasner $(m,2)$-hyperring, $Q$  a weakly $(2,2)$-absorbing $\delta$-semiprimary hyperideal of $G$ and $k(Q_1,Q_2,x) \subseteq Q$ for some $x\in G$ and two hyperideals $Q_1, Q_2$ of $G$. If $(q_1,q_2,x)$ is not a $\delta$-$(2,2)$-zero of $Q$ for all $q_1\in Q_1$ and $q_2 \in Q_2$, then $k(Q_1,x) \subseteq  \delta(Q)$ or $k(Q_2,x) \subseteq \delta(Q)$ or $k(Q_1,Q_2) \subseteq \delta(Q)$.
\end{theorem}
\begin{proof}
Let $k(Q_1,Q_2,x) \subseteq Q$, $k(Q_1,x) \nsubseteq  \delta(Q)$, $k(Q_2,x) \nsubseteq \delta(Q)$ and  $k(Q_1,Q_2) \nsubseteq \delta(Q)$. Then we get $k(q,x) \notin \delta(Q)$ and $k(q_1,Q_2) \nsubseteq \delta(Q)$ for some $q,q_1 \in Q_1$. By Theorem \ref{str}, we conclude that $k(q,Q_2) \subseteq \delta(Q)$ because $k(q,Q_2,x) \subseteq Q$, $k(q,x) \notin \delta(Q)$ and $k(Q_2,x) \nsubseteq \delta(Q)$. Also, from Theorem \ref{str}, we obtain $k(q_1,x) \in \delta(Q)$. Note that $k(h(q,q_1,0^{(m-2)}),Q_2,x)=h(k(q,Q_2,x),k(q_1,Q_2,x),0^{(m-2)}) \subseteq Q$. Then we have   $k(h(q,q_1,0^{(m-2)}),Q_2)=h(k(q,Q_2),k(q_1,Q_2),0^{(m-2)}) \subseteq \delta(Q)$ which means $k(q_1,Q_2) \subseteq  h(-k(q,Q_2),0^{(m-1)}) \subseteq \delta(Q)$ or $k(h(q,q_1,0^{(m-2)}),x)=h(k(q,x),k(q_1,x),0^{(m-2)}) \subseteq \delta(Q)$ which implies $k(q,x) \in h(-k(q_1,x),0^{(m-1)}) \subseteq \delta(Q)$. This is a contradiction. Hence $k(Q_1,x) \subseteq  \delta(Q)$ or $k(Q_2,x) \subseteq \delta(Q)$ or $k(Q_1,Q_2) \subseteq \delta(Q)$.
\end{proof}
\begin{definition}
Suppose that $G$ is a commutative Krasner $(m,2)$-hyperring and $Q_1^3,Q$ be some proper hyperideals of $G$ such that  $Q$ is a weakly $(2,2)$-absorbing $\delta$-semiprimary hyperideal of $G$.  $Q$ is said to be a free $\delta$-$(2,2)$-zero with respect to $k(Q_1^3)$ if $(q_1,q_2,q_3)$ is not a $\delta$-$(2,2)$-zero of $Q$ for every  $q_1 \in Q_1$, $q_2 \in Q_2$ and $q_3 \in Q_3$.
\end{definition}
\begin{theorem} \label{str3}
Let  $G$ be a commutative Krasner $(m,2)$-hyperring, $Q$  a weakly $(2,2)$-absorbing $\delta$-semiprimary hyperideal of $G$ and $0 \neq k(Q_1,Q_2,Q_3) \subseteq Q$ for some  hyperideals $Q_1^3$ of $G$. If $Q$ is a free $\delta$-$(2,2)$-zero with respect to $k(Q_1^3)$, then $k(Q_1^2) \subseteq  \delta(Q)$ or $k(Q_2^3) \subseteq \delta(Q)$ or $k(Q_1,Q_3) \subseteq \delta(Q)$.
\end{theorem}
\begin{proof}
Suppose that $k(Q_1^3) \subseteq Q$ but $k(Q_1^2) \nsubseteq  \delta(Q)$ or $k(Q_2^3) \nsubseteq \delta(Q)$ or $k(Q_1,Q_3) \nsubseteq \delta(Q)$.
This implies that $k(q,Q_2) \nsubseteq \delta(Q)$ and $k(q_1,Q_3) \nsubseteq \delta(Q)$ for some $q,q_1 \in Q_1$. By Theorem \ref{str2}, we get $k(q,Q_3) \subseteq \delta(Q)$ because $k(q,Q_2^3) \subseteq Q$, $k(Q_2^3) \nsubseteq \delta(Q)$ and $k(q,Q_2) \nsubseteq \delta(Q)$. Also, from Theorem \ref{str2}, we obtain  $k(q_1,Q_2) \nsubseteq \delta(Q)$ as $k(q_1,Q_2^3) \subseteq Q$, $k(Q_2^3) \nsubseteq \delta(Q)$ and $k(q_1,Q_2) \nsubseteq \delta(Q)$. Since $k(h(q,q_1),Q_2^3) \subseteq Q$ and $k(Q_2^3) \nsubseteq \delta(Q)$, we have $k(h(q,q_1,0^{(m-2)}),Q_2)=h(k(q,Q_2),k(q_1,Q_2),0^{(m-2)}) \subseteq \delta(Q)$ or $k(h(q,q_1,0^{(m-2)}),Q_3)=h(k(q,Q_3),k(q_1,Q_3),0^{(m-2)}) \subseteq \delta(Q)$. In the first case, we conclude that $k(q,Q_2) \in h(-k(q_1,Q_2),0^{(m-1)}) \subseteq \delta(Q)$, a controdiction. Moreover, the second case leads to a contradiction because $k(q_1,Q_3) \in h(-k(q,Q_3),0^{(m-1)}) \subseteq \delta(Q)$. Thus $k(Q_1^2) \subseteq  \delta(Q)$ or $k(Q_2^3) \subseteq \delta(Q)$ or $k(Q_1,Q_3) \subseteq \delta(Q)$.
\end{proof}
\begin{definition}
Assume that $Q$ is a weakly $(k,n)$-absorbing $\delta$-semiprimary hyperideal of $G$. Then $(a_1,\cdots,a_{tn-t+1})$ is called $\delta$-$(t,n)$-zero of $Q$ if $k(a_1^{tn-t+1})=0$ and none $k$-product of the terms $(t-1)n-t+2$ of $a_i^,$s is in $\delta(Q)$.
\end{definition}
\begin{theorem} \label{str4}
Assume that  $Q$ is  a weakly $(t,n)$-absorbing $\delta$-semiprimary hyperideal of $G$ and $k(a_1,\cdots,\hat{a}_{i_1},\cdots,\hat{a}_{i_2},\cdots,\hat{a}_{i_s},\cdots,a_{tn-t+1},Q_1^s) \subseteq Q$ for some $a_1^{tn-t+1}\in G$ and some hyperideals $Q_1, \cdots Q_s$ of $G$ such that $1 \leq i_1,\cdots,i_s \leq tn-t+1$ and $1 \leq s \leq (t-1)n-t+2$. If $(a_1,\cdots,\hat{a}_{i_1},\cdots,\hat{a}_{i_2},\cdots,\hat{a}_{i_s},\cdots,a_{tn-t+1},q_1^s)$ is not a $\delta$-$(t,n)$-zero of $Q$ for all $q_i\in Q_i$, then   $k$-product of $(t-1)n-t+2$ of $a_1,\cdots,\hat{a}_{i_1},\cdots,\hat{a}_{i_2},\cdots,\hat{a}_{i_s},\cdots,a_{tn-t+1},Q_1^s$ including at least one of the $Q_i^,$s is in $\delta(Q)$. 
\end{theorem}
\begin{proof}
We prove it with induction on $s$. Let us consider $s=1$. In this case we  show  that  $k$-product of $(t-1)n-t+2$ of $a_1,\cdots,\hat{a}_{i_1},\cdots,a_{tn-t+1},Q_1$ including  $Q_1$ is in $\delta(Q)$. Assume that all products of $(t-1)n-t+2$ of $a_1,\cdots,\hat{a}_{i_1},\cdots,a_{tn-t+1},Q_1$  are not in $\delta(Q)$. We consider  $k(a_2^{(t-1)n-t+2},Q_1) \notin \delta(Q)$. Since $(a_2^{tn-t+1},q_1)$ is not a $\delta$-$(t,n)$-zero of $Q$ for all $q_1 \in Q_1$,  then we conclude that $k$-product of the $(t-1)n-t+2$ of $a_i^,$s with $q_1$ is in $\delta(Q)$. By a similar argument given in the proof of Theorem \ref{str}, we have $k(a_3^{tn-t+1},h(a_1,q_1,0^{(m-2)}))=h(k(a_3^{tn-t+1},a_1),k(a_3^{tn-t+1},q_1),0^{(m-2)}) \subseteq \delta(Q)$ which implies $k(a_3^{tn-t+1},a_1) \in h(-k(a_3^{tn-t+1},q_1),0^{(m-1)}) \subseteq \delta(Q)$, a contradiction. This implies  that $k$-product of $(t-1)n-t+2$ of $a_1,\cdots,\hat{a}_{i_1},\cdots,a_{tn-t+1},Q_1$ including  $Q_1$ is in $\delta(Q)$. Now, we suppose that the claim
holds for all positive integers which are less than $s$.  Let  $k(a_1,\cdots,\hat{a}_{i_1},\cdots,\hat{a}_{i_2},\cdots,\hat{a}_{i_s},\cdots,a_{tn-t+1},Q_1^s) \subseteq Q$ but all  products of $(t-1)n-t+2$ of $a_1,\cdots,\hat{a}_{i_1},\cdots,\hat{a}_{i_2},\cdots,\hat{a}_{i_s},\cdots,a_{tn-t+1},Q_1^s$ including at least one of the $Q_i^,$s are not in $\delta(Q)$.  We may assume that $k(a_{s+1}^{tn-t+1},Q_1^s) \notin \delta(Q)$. Note that $ (a_{s+1}^{tn-t+1},q_1^s)$ is not a $\delta$-$(t,n)$-zero of $Q$ for all $q_1^s \in Q$. We get  $ k(a_{s+1}^{tn-t+1},h(a_1,q_1,0^{(m-2)}),\cdots,h(a_s,q_s,0^{(m-2)})) \subseteq \delta(Q)$ by induction hypothesis and Theorem \ref{str2}. Then we conclude that

 $\hspace{0.7cm}k(a_{s+1}^{tn-t+1},h(a_1,q_1,0^{(m-2)}),\cdots,h(a_1,\hat{q}_1,0^{(m-2)})_{i_1},\cdots,h(a_2,\hat{q}_2,0^{(m-2)})_{i_2},$
 
 $\hspace{1cm}\cdots,h(a_{n-1},\hat{q}_{n-1},0^{(m-2)})_{i_{n-1}},\cdots,h(a_s,q_s,0^{(m-2)})) \subseteq \delta(Q)$\\ or 
 
 $\hspace{0.7cm}k(a_{s+1},\cdots,\hat{a}_{i_{s+1}},\cdots,\hat{a}_{i_{s+2}},\cdots,\hat{a}_{i_{s+n-1}},\cdots,a_{tn-t+1},h(a_1,q_1,0^{(m-2)}),$
 
 $\hspace{1cm}\cdots,h(a_s,q_s,0^{(m-2)})) \subseteq \delta(Q)$\\ for some $i \in \{1,\cdots,s\}$. This implies that $k(a_{s+1},\cdots,a_{tn-t+1},\cdots,a_n,\cdots,a_s) \in \delta(Q)$ or $k(a_{s+n},\cdots,a_{tn-t+1},\cdots,a_1^s) \in \delta(Q)$, a contradiction. Then we conclude that $k$-product of $(t-1)n-t+2$ of $a_1,\cdots,\hat{a}_{i_1},\cdots,\hat{a}_{i_2},\cdots,\hat{a}_{i_s},\cdots,a_{tn-t+1},Q_1^s$ including at least one of the $Q_i^,$s is in $\delta(Q)$. 
\end{proof}
\begin{definition}
Suppose that $Q_1^n,Q$ be some proper hyperideals of $G$ such that  $Q$ is a weakly $(t,n)$-absorbing $\delta$-semiprimary hyperideal of $G$ and $k(Q_1^{tn-t+1}) \subseteq Q$.  $Q$ is called  a free $\delta$-$(t,n)$-zero with respect to $k(Q_1^{tn-t+1})$ if $(q_1,\cdots,q_{tn-t+1})$ is not a $\delta$-$(t,n)$-zero of $Q$ for every  $q_i \in Q_i$ with $1 \leq i \leq tn-t+1$.
\end{definition}
\begin{theorem} \label{azad}
Assume that  $Q$ is  a weakly $(t,n)$-absorbing $\delta$-semiprimary hyperideal of $G$ and $0 \neq k(Q_1^{tn-t+1}) \subseteq Q$ for   some hyperideals $Q_1^{tn-t+1}$ of $G$. If  $Q$ is a free $\delta$-$(t,n)$-zero with respect to $k(Q_1^{tn-t+1})$, then   $k$-product of $(t-1)n-t+2$ of the $Q_i$  is a subset of $\delta(Q)$. 
\end{theorem}
\begin{proof}
This can be proved by Theorem \ref{str4}, in a very similar manner to the way in which Theorem \ref{str3} was proved.
\end{proof}
\begin{theorem}\label{6}
Let  $G$ be a commutative Krasner $(m,2)$-hyperring and $Q$ be a weakly $(2,2)$-absorbing $\delta$-semiprimary hyperideal of $G$. If $(x,y,z)$ is  an $\delta$-$(2,2)$-zero of $Q$ for some $x,y,z \in G$, then 
\begin{itemize} 
\item[\rm(i)]~$k(x,y,Q)=k(y,z,Q)=k(x,z,Q)=0$
\item[\rm(ii)]~$k(x,Q^{(2)})=k(y,Q^{(2)})=k(z,Q^{(2)})=0$
\end{itemize}
\end{theorem}
\begin{proof}
(i) Let $Q$ be a weakly $(2,2)$-absorbing $\delta$-semiprimary hyperideal of $G$ and $(x,y,z)$ be  an $\delta$-$(2,2)$-zero of $Q$. Let us assume that $k(x,y,Q) \neq 0$. This means that $k(x,y,q) \neq 0$ for some $q \in Q$. So we have $0 \neq k(x,h(z,q,0^{(m-2)}),y)=h(k(x,z,y),k(x,q,y),0^{(m-2)}) \subseteq Q$. Since $Q$ is  weakly $(2,2)$-absorbing $\delta$-semiprimary and $k(x,y) \notin \delta(Q)$, we get $k(x,h(z,q,0^{(m-2)}))=h(k(x,z),k(x,q),0^{(m-2)}) \subseteq \delta(Q)$ or $k(h(z,q,0^{(m-2)}),y)=h(k(z,y),k(q,y),0^{(m-2)}) \subseteq \delta(Q)$. In the first case, we have $k(x,z) \in h(-k(x,q),0^{(m-1)}) \subseteq \delta(Q)$ which is a contradiction. The second case leads to a contradiction because $ k(z,y) \in h(-k(q,y),0^{(m-1)}) \subseteq \delta(Q)$. Thus $k(x,y,Q)=0$. Similiar for the other cases. 

(ii) Let $k(x,Q^{(2)}) \neq 0$. This implies that $k(x,q_1^2) \neq 0$ for some $q_1, q_2 \in Q$. Therefore 

$\hspace{2.7cm}0 \neq k(x,h(y,q_1,0^{(m-2)}),h(z,q_2,0^{(m-2)}))$

$\hspace{3cm}=h(k(x,y,z),k(x,y,q_2),k(x,q_1,z),k(x,q_1^2),0^{(m-4)})$

$\hspace{3cm} \subseteq Q$. \\
Since $Q$ is a weakly $(2,2)$-absorbing $\delta$-semiprimary hyperideal of $G$, we obtain the following cases:
\begin{itemize} 
\item[\rm Case 1.]~ $k(x,h(y,q_1,0^{(m-2)})) \subseteq \delta(Q)$ which implies $h(k(x,y),k(x,q_1),0^{(m-2)}) \subseteq \delta(Q)$. Then we have 
$k(x,y) \in h(-k(x,q_1),0^{(m-1)}) \subseteq \delta(Q)$, a contradiction.
\item[\rm Case 2.]~ $k(x,h(z,q_2,0^{(m-2)})) \subseteq \delta(Q)$ which means $h(k(x,z),k(x,q_2),0^{(m-2)}) \subseteq \delta(Q)$. This follows that $k(x,z) \in h(-k(x,q_2),0^{(m-1)}) \subseteq \delta(Q)$, a contradiction.
\item[\rm Case 3.]~ $k(h(y,q_1,0^{(m-2)}),h(z,q_2,0^{(m-2)})) \subseteq \delta(Q)$ and so 

$h(k(y,z),k(q_1,z),k(y,q_2),k(q_1^2),0^{(m-4)}) \subseteq \delta(Q)$.\\ This implies that $k(y,z) \in h(-k(q_1,z),-k(y,q_2),-k(q_1^2), 0^{(m-2)}) \subseteq \delta(Q)$ which is a contradiction. 
\end{itemize}
Therefore $k(x,Q^{(2)}) = 0$. Similiar for the other cases.
\end{proof}
\begin{theorem}\label{7}
Let  $G$ be a commutative Krasner $(m,2)$-hyperring and $Q$ be a weakly $(2,2)$-absorbing $\delta$-semiprimary hyperideal of $G$ but is not $(2,2)$-absorbing $\delta$-semiprimary. Then $k(Q^{(3)})=0$.
\end{theorem}
\begin{proof}
Let $Q$ be  a weakly $(2,2)$-absorbing $\delta$-semiprimary hyperideal of $G$ but is not $(2,2)$-absorbing $\delta$-semiprimary. This implies that we have  an $\delta$-$(2,2)$-zero of $Q$ for some $x,y,z \in G$. Let us assume that $k(Q^{(3)}) \neq 0$. Then  $k(q_1^3) \neq 0$ for some $q_1^3 \in Q$. Therefore we have

 $k(h(x,q_1,0^{(m-2)}),h(y,q_2,0^{(m-2)}),h(z,q_3,0^{(m-2)}))$
 
$\hspace{0.2cm}=h(h(h(k(x,y,z),k(q_1,y,z),0^{(m-2)}),h(k(x,y,q_3),k(q_1,y,q_3),0^{(m-2)})),$

$\hspace{0.8cm}h(h(k(q_2,x,z),h(q_1^2,z),0^{(m-2)})),h(h(k(q_1^3),k(x,q_2^3)),0).$\\
From $k(q_1^3) \neq 0$, it follows that 

$\hspace{1cm}0 \neq k(h(x,q_1,0^{(m-2)}),h(y,q_2,0^{(m-2)}),h(z,q_3,0^{(m-2)})) \subseteq Q$ \\
by Theorem \ref{6}. Since $Q$ is   weakly $(2,2)$-absorbing $\delta$-semiprimary, we have  $k(h(x,q_1,0^{(m-2)}),h(y,q_2,0^{(m-2)})) \subseteq \delta(Q)$ or $k(h(x,q_1,0^{(m-2)}),h(z,q_3,0^{(m-2)})) \subseteq \delta(Q)$ or $k(h(x,q_1,0^{(m-2)}),h(z,q_3,0^{(m-2)})) \subseteq \delta(Q)$. In the first possibilty, we obtain $h(k(x,y),k(x,q_2),k(q_1,y),k(q_1^2),0^{(m-4)}) \subseteq \delta(Q)$ which means $k(x,y) \in h(-k(x,q_2),-k(q_1,y),-k(q_1^2),0^{(m-3)}) \subseteq \delta(Q)$ which is a contradiction. Moreover, the other possibilities lead to a contradiction. Thus $k(Q^{(3)})=0$.
\end{proof}
\begin{definition}
Assume that $Q$ is a weakly $(k,n)$-absorbing $\delta$-semiprimary hyperideal of $G$. Then $(a_1,\cdots,a_{tn-t+1})$ is called $\delta$-$(t,n)$-zero of $Q$ if $k(a_1^{tn-t+1})=0$ and none $k$-product of the terms $(t-1)n-t+2$ of $a_i^,$s is in $\delta(Q)$.
\end{definition}
\begin{theorem}\label{8}
If $Q$ is a weakly $(t,n)$-absorbing $\delta$-semiprimary hyperideal of $G$ and $(a_1,\cdots,a_{tn-t+1})$ is a $\delta$-$(t,n)$-zero of $Q$, then  for $1 \leq i_1,\cdots,i_s \leq tn-t+1$ and $1 \leq s \leq (t-1)n-t+2$,
\[k(a_1,\cdots,\hat{a}_{i_1},\cdots,\hat{a}_{i_2},\cdots,\hat{a}_{i_s},\cdots,a_{tn-t+1},Q^{(s)})=0.\]
\end{theorem}
\begin{proof}
We use the induction on $s$. Assume that $s=1$. Let us suppose that $k(a_1,\cdots,\hat{a}_{i_1},\cdots,a_{tn-t+1},Q) \neq 0$. We may assume that $k(a_2^{tn-t+1},Q) \neq 0$. Therefore $k(a_2^{tn-t+1},q) \neq 0$ for some $q \in Q$. Hence every $k$-product of the $(t-1)n-t+2$ of $a_i^,$s including $q$ is in $\delta(Q)$. By the same argument given in Theorem \ref{6}, we have $k(a_3^{tn-t+1},h(a_1,q,0^{(m-2)}))=h(k(a_3^{tn-t+1},a_1),k(a_3^{tn-t+1},q),0^{(m-2)}) \subseteq \delta(Q)$ which implies $k(a_3^{tn-t+1},a_1) \in h(-k(a_3^{tn-t+1},q),0^{(m-1)}) \subseteq \delta(Q)$, a contradiction. This means that $k(a_1,\cdots,\hat{a}_{i_1},\cdots,a_{tn-t+1},Q) = 0$. Now, let us suppose that $k(a_1,\cdots,\hat{a}_{i_1},\cdots,\hat{a}_{i_2},\cdots,\hat{a}_{i_s},\cdots,a_{tn-t+1},Q^{(s)}) \neq 0$. We may assume that $k(a_{s+1}^{tn-t+1},Q^{(s)}) \neq 0$. Hence $0 \neq k(a_{s+1}^{tn-t+1},q_1^s) \in Q$ for some $q_1^s \in Q$. It follows that  $0 \neq k(a_{s+1}^{tn-t+1},h(a_1,q_1,0^{(m-2)}),\cdots,h(a_s,q_s,0^{(m-2)})) \subseteq Q$ by Theorem \ref{6} and induction hypothesis. Then we conclude that

 $\hspace{0.7cm}k(a_{s+1}^{tn-t+1},h(a_1,q_1,0^{(m-2)}),\cdots,h(a_1,\hat{q}_1,0^{(m-2)})_{i_1},\cdots,h(a_2,\hat{q}_2,0^{(m-2)})_{i_2},$
 
 $\hspace{1cm}\cdots,h(a_{n-1},\hat{q}_{n-1},0^{(m-2)})_{i_{n-1}},\cdots,h(a_s,q_s,0^{(m-2)})) \subseteq \delta(Q)$\\ or 
 
 $\hspace{0.7cm}k(a_{s+1},\cdots,\hat{a}_{i_{s+1}},\cdots,\hat{a}_{i_{s+2}},\cdots,\hat{a}_{i_{s+n-1}},\cdots,a_{tn-t+1},h(a_1,q_1,0^{(m-2)}),$
 
 $\hspace{1cm}\cdots,h(a_s,q_s,0^{(m-2)})) \subseteq \delta(Q)$\\ for some $i \in \{1,\cdots,s\}$. This implies that $k(a_{s+1},\cdots,a_{tn-t+1},\cdots,a_n,\cdots,a_s) \in \delta(Q)$ or $k(a_{s+n},\cdots,a_{tn-t+1},\cdots,a_1^s) \in \delta(Q)$, a contradiction. Thus we conclude that  $k(a_1,\cdots,\hat{a}_{i_1},\cdots,\hat{a}_{i_2},\cdots,\hat{a}_{i_s},\cdots,a_{tn-t+1},Q^{(s)})=0$.
\end{proof}
\begin{theorem} \label{9}
Let $Q$ be a weakly $(t,n)$-absorbing $\delta$-semiprimary hyperideal of $G$ but is not $(t,n)$-absorbing $\delta$-semiprimary. Then $k(Q^{(tn-t+1)})=0$.
\end{theorem}
\begin{proof}
Assume that  $Q$ is a weakly $(t,n)$-absorbing $\delta$-semiprimary hyperideal of $G$ but is not $(t,n)$-absorbing $\delta$-semiprimary. Then there exists a $\delta$-$(t,n)$-zero $(a_1,\cdots,a_{tn-t+1})$ of $Q$. Now, the claim follows by using Theorem \ref{8}, in a very similar manner to the way in which Theorem \ref{7} was proved.
\end{proof}
As an instant consequence of the previous theorem, we have the following explicit results.

\begin{corollary}
Let $Q$ be a weakly $(t,n)$-absorbing $\delta$-semiprimary hyperideal of $G$ but is not $(t,n)$-absorbing $\delta$-semiprimary. Then $Q \subseteq rad(0)$.
\end{corollary}
\begin{corollary}
Assume that the commutative Krasner $(m,n)$-hyperring $G$  has no non-zero nilpotent elements. If $Q$ is a weakly $(t,n)$-absorbing $\delta$-semiprimary hyperideal of $G$, then $Q$ is an $(t,n)$-absorbing $\delta$-semiprimary hyperideal of $G$.
\end{corollary}
The next theorem provides us how to determine weakly $(t,n)$-absorbing $\break$ $\delta$-semiprimary hyperideal  to be $(t,n)$-absorbing  $\delta$-semiprimary.
\begin{theorem}
Let $Q$ be a weakly $(t,n)$-absorbing $\delta$-semiprimary hyperideal of $G$ such that $\delta(Q)=\delta(0)$. Then $Q$ is not $(t,n)$-absorbing $\delta$-semiprimary if and only if there exists a $\delta$-$(t,n)$-zero of $0$.
\end{theorem}
\begin{proof}
$\Longrightarrow $ Assume that $Q$ is not an $(t,n)$-absorbing $\delta$-semiprimary hyperideal of $G$. This implies that $k(a_1^{tn-t+1})=0$ and none $k$-product of the terms $(t-1)n-t+2$ of $a_i^,$s is in $\delta(Q)$ for some $a_1^{tn-t+1} \in G$. From $\delta(Q)=\delta(0)$, it follows that $(a_1^{tn-t+1})$ is a $\delta$-$(t,n)$-zero of $0$.

$\Longleftarrow$ Straightforward
\end{proof}

Let $(G_1,h_1,k_1)$ and  $(G_2,h_2,k_2)$ be two commutative Krasner $(m, n)$-hyperrings. Recall from \cite{d1} that a mapping
$f : G_1 \longrightarrow G_2$ is called a homomorphism if  we have
$f(h_1(a_1^m)) = h_2(f(a_1),\cdots,f(a_m))$
and
$f(k_1(b_1^n)) = k_2(f(b_1),...,f(b_n)) $ for all $a^m _1 \in G_1$ and $b^n_ 1 \in G_1$.
 Let $\delta$ and $\delta^{\prime}$ be hyperideal expansions   of $G_1$ and $G_2$, respectively. Recall from \cite{mah3} that $f: G_1 \longrightarrow G_2$ is called a $\delta \delta^{\prime}$-homomorphism if $\delta(f^{-1}(Q_2)) = f^{-1}(\delta^{\prime}(Q_2))$ for  hyperideal $Q_2$ of
$G_2$. 
 Note that $\delta^{\prime}(h(Q_1)=h(\delta(Q_1)$ for   $\delta \delta^{\prime}$-epimorphism $f$ and   for hyperideal $Q_1$ of $G_1$ with $ Ker (f) \subseteq Q_1$.
\begin{theorem} \label{home}
Let $(G_1,h_1,k_1)$ and $(G_2,h_2,k_2)$ be two Krasner $(m,n)$-hyperrings and $ f:G_1 \longrightarrow G_2$ be a $\delta\delta^{\prime}$-homomorphism. Then the followings hold: 
\begin{itemize}
\item[\rm(i)]~ If $Q_2$ is an $(t,n)$-absorbing $\delta^{\prime}$-semiprimary   hyperideal of $G_2$, then $f^{-1} (Q_2)$ is an $(t,n)$-absorbing $\delta$-semiprimary   hyperideal of $G_1$.
\item[\rm(ii)]~ If $Q_2$ is a weakly $(t,n)$-absorbing $\delta^{\prime}$-semiprimary   hyperideal of $G_2$ and $Ker f$ is a weakly $(t,n)$-absorbing $\delta$-semiprimary   hyperideal of $G_1$, then $f^{-1} (Q_2)$ is a weakly $(t,n)$-absorbing $\delta$-semiprimary   hyperideal of $G_1$.

 \item[\rm(iii)]~Let  $f $ be an epimorphism and $Q_1$ be a proper hyperideal of $G_1$ containing $Ker f$. If $Q_1$ is a (weakly) $(t,n)$-absorbing $\delta$-semiprimary hyperideal of $G_1$, then $f(Q_1)$ is a (weakly) $\delta^{\prime}$-semiprimary  hyperideal of $G_2$. 
 \end{itemize}
\end{theorem}
\begin{proof}
$(i)$ Let  $k_1(a_1^{tn-t+1}) \in f^{-1}(Q_2)$ for $a_1^{kn-k+1} \in G_1$. Then we get $f(k_1(a_1^{tn-t+1}))=k_2(f(a_1),...,f(a_{tn-t+1})) \in Q_2$. Since $Q_2$ is an $(t,n)$-absorbing $\delta^{\prime}$-semiprimary   hyperideal of $G_2$, then there exist $(t-1)n-t+2$ of $f(a_i)^,$s whose $k_2$-product is an element in $\delta^{\prime}(Q_2)$. It follows that the image $f$ of $(t-1)n-t+2$ of $a_i^,$ whose $k_2$-product is in $\delta^{\prime}(Q_2)$ which means there exist $(t-1)n-t+2$ of $a_i^,$ whose $k_1$-product is in $f^{-1}(\delta^{\prime}(Q_2))=\delta(f^{-1}(Q_2))$. Thus $f^{-1}(Q_2)$ is an $(t,n)$-absorbing $\delta$-semiprimary   hyperideal of $G_1$.

$(ii)$ Assume that  $k_1(a_1^{tn-t+1}) \in f^{-1}(Q_2)$ for $a_1^{kn-k+1} \in G_1$. Therefore $f(k_1(a_1^{tn-t+1}))=k_2(f(a_1),...,f(a_{tn-t+1})) \in Q_2$. If $0 \neq f(k_1(a_1^{tn-t+1}))$, then it can be proved
by using an argument similar to that in the proof of the part (i). Let us assume that $ f(k_1(a_1^{tn-t+1})) =0$. Then we obtain $k_1(a_1^{tn-t+1}) \in Ker f$. Since $Ker f$ is a weakly  $(t,n)$-absorbing $\delta$-semiprimary   hyperideal of $G_1$, then there exist $(t-1)n-t+2$ of $a_i^,$s whose $k_1$-product is an element in $\delta(Ker f)$. From $\delta(Ker f)=\delta(f^{-1}(0)) \subseteq \delta(f^{-1}(Q_2))$, it follows that $f^{-1}(Q_2)$ is a weakly $(t,n)$-absorbing $\delta$-semiprimary   hyperideal of $G_1$.
 
 $(iii)$ Let ($0 \neq k_2(b_1^{tn-t+1}) \in f(Q_1)$) $k_2(b_1^{tn-t+1}) \in f(Q_1)$ for some $b_1^{tn-t+1} \in G_2$. Since $f $ be an epimorphism, then there exist $a_i \in G_1$ for each $1 \leq i \leq tn-t+1$ such that $f(a_i)=b_i$. Hence $k_2(b_1^{tn-t+1})=k_2(f(a_1),\cdots,f(a_{tn-t+1}))=f(k_1(a_1^{tn-t+1})) \in f(Q_1)$. Since $Q_1$ containing $Ker f$, we conclude that ($0 \neq k_1(a_1^{tn-t+1}) \in Q_1$) $ k_1(a_1^{tn-t+1}) \in Q_1$. As $Q_1$ is a (weakly) $(t,n)$-absorbing $\delta$-semiprimary hyperideal of $G_1$, then there exist $(t-1)n-t+2$ of $a_i^,$s whose $k_1$-product is in $\delta(Q_1)$. Now, since $f$ is a homomorphism and $f(\delta(Q_1))=\delta^{\prime}(f(Q_1))$,  the proof is
completed.
\end{proof}

Let $P$ be a hyperideal of  $(G, h, k)$. Then the set
$G/P = \{h(g^{i-1}_1, P, g^m_{i+1}) \ \vert \  g^{i-1}_1,g^m_{i+1} \in G \}$ with $h$ and $k$ which are defind by

$\hspace{2cm}h(h(g_{11}^{1 (i-1)}, P, g^{1m}_ {1(i+1)}),..., h(g_{m1}^{ m(i-1)}, P, g^{mm}_ {m(i+1)}))$ 

$\hspace{2.5cm}= h (h(g^{m1}_{11}),..., h(g^{m(i-1)}_{1(i-1)}), P, h(g^{m(i+1)}_{1(i+1)} ),..., h(g^{mm}_ {1m}))$\\
and

$\hspace{2cm}k(h(g_{11}^{1 (i-1)}, P, g^{1m}_ {1(i+1)}),..., h(g_{n1}^{ n(i-1)}, P, g^{nm}_ {n(i+1)}))$ 

$\hspace{2.5cm}= h (k(g^{n1}_{11}),..., k(g^{n(i-1)}_{1(i-1)}), P, k(g^{n(i+1)}_{1(i+1)} ),..., k(g^{nm}_ {1m}))$\\
 for all $g_{11}^{1m},...,g_{m1}^{mm} \in G$ and $g_{11}^{1m},...,g_{n1}^{nm} \in G$, construct a Krasner $(m, n)$-hyperring \cite{sorc1}. 
 \begin{theorem}
Let $P$ and $Q$ be two proper hyperideals of $G$ with $P \subseteq Q$. If $Q$ is an  $(t,n)$-absorbing $\delta$-semiprimary hyperideal of $G$, then $Q/P$ is an $(t,n)$-absorbing $\delta_q$-semiprimary hyperideal of $G/P$.
\end{theorem}
\begin{proof}
By considering the  natural homomorphism $\pi : G \longrightarrow G/P$, defined by $\pi(a)=f(a,P,0^{(m-2)})$ and using Theorem \ref{home}, we are done.
\end{proof}
\begin{theorem}
Let $Q$ be an  $(t,n)$-absorbing $\delta$-semiprimary hyperideal of $G$. If $G^{\prime}$ is a subhyperring of $G$ such that $G^{\prime} \nsubseteq Q$, then $Q \cap G^{\prime}$ is an  $(t,n)$-absorbing $\delta$-semiprimary hyperideal of $G^{\prime}$. 
\end{theorem}
\begin{proof}
It follows by Theorem \ref{home}.
\end{proof}
\begin{theorem} 
 Let $\delta_1$ and $\delta_2$ be two hyperideal expansions of Krasner $(m,n)$-hyperrings $G_1$ and $G_2$, respectively,   such that $\delta(Q_1 \times  Q_2)=\delta_1(Q_1) \times  \delta_2(Q_2)$ for  hyperideals $Q_1$ and $Q_2$ of $G_1$ and $G_2$, respectively. If $Q=Q_1 \times G_2$ is a weakly  $(t,n)$-absorbing $\delta$-semiprimary hyperideal of $G_1  \times G_2$,
 then it is an $(t,n)$-absorbing $\delta$-semiprimary hyperideal of $G_1  \times G_2$.
\end{theorem}
\begin{proof}
Assume that $Q_1 \times G_2$ is a weakly  $(t,n)$-absorbing $\delta$-semiprimary hyperideal of $G_1  \times G_2$. Since $k(Q^{(tn-t+1)}) \neq 0$, we conclude that $Q=Q_1 \times G_2$ is an $(t,n)$-absorbing $\delta$-semiprimary hyperideal of $G_1  \times G_2$ by Theorem \ref{9}.
\end{proof}
We say that $\delta$  has $(\mathfrak{P})$ property  if it satisfies the  condition:  $\delta(Q)=G$ if and only if $Q=G$ for all hyperideals $Q$ of $G$.
\begin{theorem} \label{cart2}
 Let $\delta_1,\cdots, \delta_{tn-t+1}$ be hyperideal expansions of Krasner $(m,n)$-hyperrings $G_1,\cdots,G_{tn-t+1}$  such that each $\delta_i$ has $(\mathfrak{P})$ property and  $\delta(Q_1 \times \cdots \times Q_{tn-t+1})=\delta_1(Q_1) \times \cdots \times \delta_{tn-t+1}(Q_{tn-t+1})$ for  hyperideals $Q_1,\cdots,Q_{tn-t+1}$ of $G_1,\cdots,G_{tn-t+1}$, respectively. If $Q=Q_1 \times \cdots \times Q_{tn-t+1}$ is a weakly $(t,n)$-absorbing $\delta$-semiprimary hyperideal of $G=G_1 \times ... \times G_{tn-t+1}$, then $Q$ is an $(t,n)$-absorbing $\delta$-semiprimary hyperideal of $G=G_1 \times ... \times G_{tn-t+1}$.
 \end{theorem}
 \begin{proof}
 Let $Q$ is a weakly $(t,n)$-absorbing $\delta$-semiprimary hyperideal of $G$. Let us consider the following elements of $G$: $x_i=(1_{G_1}.\cdots,1_{G_{i-1}},a_i,1_{G_{i+1}},\cdots,1_{G_{tn-t+1}})$ for all $1 \leq i \leq tn-t+1$. Then we have $0 \neq k(x_1^{tn-t+1}) \in Q$. Since $Q=Q_1 \times \cdots \times Q_{tn-t+1}$ is a weakly $(t,n)$-absorbing $\delta$-semiprimary hyperideal of $G=G_1 \times ... \times G_{tn-t+1}$, then there exists $(t-1)n-t+2$ of the $x_i^,$s whose $k$-product is in $\delta(Q)=\delta_1(Q_1) \times \cdots \times \delta_{tn-t+1}(Q_{tn-t+1})$. This implies that there exists some $1 \leq j \leq tn-t+1$ such that $1_{G_j} \in \delta_j(Q_j)$ which means $\delta_j(Q_j)=G_j$. Since $\delta_j$ has $(\mathfrak{P})$ property, then $Q_j=G_j$. Hence we conclude that $k(Q^{(tn-t+1)}) \neq 0$ which implies $Q$ is an $(t,n)$-absorbing $\delta$-semiprimary hyperideal of $G$ by Theorem \ref{9}.
 \end{proof}
 \begin{theorem}
 Let $\delta_1,\cdots, \delta_{tn-t+1}$ be hyperideal expansions of Krasner $(m,n)$-hyperrings $G_1,\cdots,G_{tn-t+1}$  such that each $\delta_i$ has $(\mathfrak{P})$ property and $\delta(Q_1 \times \cdots \times Q_{tn-t+1})=\delta_1(Q_1) \times \cdots \times \delta_{tn-t+1}(Q_{tn-t+1})$ for  hyperideals $Q_1,\cdots,Q_{tn-t+1}$ of $G_1,\cdots,G_{tn-t+1}$, respectively. If $Q=Q_1 \times \cdots \times Q_{tn-t+1}$ is a weakly $(t+1,n)$-absorbing $\delta$-semiprimary hyperideal of $G=G_1 \times ... \times G_{tn-t+1}$, then either there exists $1 \leq u \leq tn-t+1$ such that $Q_u$ is an $(t+1,n)$-absorbing $\delta$-semiprimary hyperideal of $G_u$ and $Q_i=G_i$ for each $1 \leq i \leq tn-t+1$ and $i \neq u$ or $Q_u$ and $Q_v$ are $(t,n)$-absorbing $\delta_{u,v}$-semiprimary hyperideals of $G_u$ and $G_v$, respectively, for some $u,v \in \{1,\cdots, tn-t+1\}$ and $Q_i=G_i$ for all $1 \leq i \leq tn-t+1$ but $i \neq u,v$. 
 \end{theorem}
\begin{proof}
Let $Q=Q_1 \times \cdots \times Q_{tn-t+1}$ be a weakly $(t+1,n)$-absorbing $\delta$-semiprimary hyperideal of $G=G_1 \times ... \times G_{tn-t+1}$. Therefore we conclude that  $Q$ is an $(t+1,n)$-absorbing $\delta$-semiprimary hyperideal of $G$ by Theorem \ref{cart2}. Now, by using Theorem \ref{cart3}, we are done.
\end{proof}

\section{conclusion}
In this paper, our purpose was to study the structure of  $(t,n)$-absorbing $\delta$-semiprimary hyperideals  which is more general than    $\delta$-primary hyperideals. Additionally, we generalized the notion to weakly $(t,n)$-absorbing $\delta$-semiprimary hyperideals.
We gave many special  results illustrating the structures.  Indeed, this paper makes a major contribution to classify hyperideals in Krasner $(m,n)$-hyperrings.

\vspace{0.5cm}
 
{\bf  Conflicts of Interest} 

The authors declare that they have no conflicts of interest.



\begin{thebibliography}{99}


\bibitem{sorc1}
R. Ameri, M. Norouzi, Prime and primary hyperideals in Krasner $(m,n)$-hyperrings,  {\it European Journal Of Combinatorics}, (2013) 379-390.




\bibitem{mah3}
M. Anbarloei, Unifing the prime and primary hyperideals under
one frame in a Krasner $(m,n)$-hyperring, {\it Communications in Algebra}, (2021) DOI: 10.1080/00927872.2021.1897988.


\bibitem{n100} 
A. Badawi, On 2-absorbing ideals of commutative rings, {\it Bull. Austral. Math. Soc.,} {\bf 75}(2007) 417-429.

\bibitem{Celikel}
E.Y. Celikel, 2-absorbing $\delta$-semiprimary ideals of commutative rings, {\it KYUNGPOOK Math. J.,} {\bf 61} (2021) 711-725.





\bibitem{s9}
B. Davvaz, T. Vougiouklis, n-ary hypergroups, {\it Iran. J. Sci. Technol.,} {\bf 30} (A2) (2006) 165-174. 

\bibitem{cartis}
B. Davvaz, Fuzzy Krasner $(m,n)$-hyperrings, {\it Computers and Mathematics with Applications,} {\bf 59} (2010) 3879-3891.


\bibitem{s6}
W. Dorente, Untersuchungen über einen verallgemeinerten Gruppenbegriff, {\it Math. Z.,} {\bf 29} (1928) 1-19. 


\bibitem{rev2}
K. Hila, K. Naka,  B. Davvaz,  On $(k,n)$-absorbing
hyperideals in Krasner $(m,n)$-hyperrings, {\it Quarterly Journal of
Mathematics}, {\bf 69}
(2018) 1035-1046.



\bibitem{s5}
E. Kasner, An extension of the group concept (reported by L.G. Weld), {\it Bull. Amer. Math. Soc.,} {\bf 10} (1904) 290-291.

\bibitem{l1}
V. Leoreanu, Canonical n-ary hypergroups, {\it Ital. J. Pure Appl. Math.,} {\bf 24}(2008).

\bibitem{l2}
V. Leoreanu-Fotea, B. Davvaz, n-hypergroups and binary relations, {\it European J. Combin.,} {\bf 29} (2008) 1027-1218.

\bibitem{l3}
 V. Leoreanu-Fotea, B. Davvaz, Roughness in n-ary hypergroups, {\it Inform. Sci.,} {\bf 178} (2008) 4114-4124. 
 
 
 

\bibitem{ma}
X. Ma, J. Zhan, B. Davvaz, Applications of rough soft sets to Krasner $(m,n)$-hyperrings and corresponding decision making methods, {\it Filomat}, {\bf 32} (2018) 6599-6614.




\bibitem{d1}
S. Mirvakili, B. Davvaz, Relations on Krasner $(m,n)$-hyperrings, {\it European J. Combin.,} {\bf 31}(2010) 790-802.




\bibitem{rev1}
S. Ostadhadi-Dehkordi, B. Davvaz,  A Note on
Isomorphism Theorems of Krasner $(m,n)$- hyperrings, {\it Arabian Journal of
Mathematics,} {\bf 5} (2016) 103-115.





\bibitem{s10}
M.M. Zahedi, R. Ameri, On the prime, primary and maximal subhypermodules, {\it Ital. J. Pure Appl. Math.,} {\bf 5} (1999) 61-80.

\end{thebibliography}
\end{document}